\documentclass{amsart}
\usepackage{youngtab}
\usepackage{graphicx,tikz}
\usepackage{mathrsfs,hyperref}
\usepackage{cancel, fullpage}

\newtheorem{theorem}{Theorem}[section]

\newtheorem{conjecture}[theorem]{Conjecture}

\newtheorem{lemma}[theorem]{Lemma}
\newtheorem{proposition}[theorem]{Proposition}
\newtheorem{corollary}[theorem]{Corollary}

\theoremstyle{definition}
\newtheorem{definition}[theorem]{Definition}

\newtheorem{remark}[theorem]{Remark}

\title{The transition matrix between the Specht and web bases is unipotent with additional vanishing entries}

\author{Heather M. Russell}
\address{Department of Mathematics and Computer Science, University of Richmond, 28 Westhampton Way, Richmond, VA 23173 U.S.A.}
\email{hrussell@richmond.edu}

\author{Julianna S. Tymoczko}
\address{Department of Mathematics and Statistics, Smith College, Northampton, MA 01063 U.S.A.}
\email{jtymoczko@smith.edu}

\begin{document}

\maketitle

\begin{abstract}
We compare two important bases of an irreducible representation of the symmetric group: the web basis and the Specht basis.  The web basis has its roots in the Temperley-Lieb algebra and knot-theoretic considerations.  The Specht basis is a classic algebraic and combinatorial construction of symmetric group representations which arises in this context through the geometry of varieties called Springer fibers.  We describe a graph that encapsulates combinatorial relations between each of these bases, prove that there is a unique way (up to scaling) to map the Specht basis into the web representation, and use this to recover a result of Garsia-McLarnan that the transition matrix between the Specht and web bases is upper-triangular with ones along the diagonal.  We then strengthen their result to prove vanishing of certain additional entries unless a nesting condition on webs is satisfied. In fact we conjecture that the entries of the transition matrix are nonnegative and are nonzero precisely when certain directed paths exist in the web graph.
\end{abstract}

\section{Introduction}

In this paper we study two closely-related pairs of objects: 1) graphs equipped with an edge-labeling that represents the action of certain permutations, and 2) natural bases for representations of the symmetric group.  Using properties of the graphs, we prove that the bases are related by a transition matrix that is upper-triangular with ones along the diagonal.  Our presentation is combinatorial, though the underlying motivation is geometric and representation-theoretic (see below, or for more detail \cite{cautis2008knot, FKK, khovanov2004crossingless, SCHÄFER_2012, Stroppel-Webster}).

The vertices of the graphs we study are classical combinatorial objects.  In one graph, the vertices are {\em standard Young tableaux} of shape $(n,n)$: in other words they are ways to fill a $2 \times n$ grid with the integers $1,2,\ldots,2n$ so that numbers increase left-to-right in each row and top-to-bottom in each column.  Standard Young tableaux are fundamental objects not just in combinatorics but also in representation theory and  geometry (see, e.g., Fulton's book for an overview \cite{MR1464693}).  They arise in our context as a geometrically-natural characterization of the cells in a topological decomposition of a family of varieties called {\em Springer fibers} \cite{MR1994220, MR2869443, MR2801314, wehrli2009remark}.  

In the second graph, the vertices are {\em webs}. Webs are drawn as arcs connecting pairs of integers $1 \leq i < j \leq 2n$ on the number line.  One other name is {\em noncrossing matchings}; as this suggests, these arcs must not intersect each other. In the literature, the webs we have just described are also called $\mathfrak{sl}_2$ webs or Temperley-Lieb diagrams, and have been studied extensively \cite{jones1983index, kauffman1987state, lickorish1992calculations, RTW}.  In general, there are webs for each $\mathfrak{sl}_k$ that diagrammatically encode the representation theory of $U(\mathfrak{sl}_k)$ and its quantum deformation (cf. \cite{cautis2014webs}). In the quantum context, webs give rise to knot and link invariants including the Jones polynomial \cite{jones1985polynomial, reshetikhin1991invariants}. Webs for other Lie types have also been studied \cite{K}.

This paper focuses on the algebraic applications of webs and tableaux rather than the geometric applications.  Consider webs as a formal basis of a complex vector space.  This vector space carries a natural $S_{2n}$-action coming from the category of $U_q(\mathfrak{sl}_2)$ representations (described further below).  The resulting web representation is irreducible and in fact is the same as the irreducible $S_{2n}$-representation on the classic Specht module of shape $(n,n)$, whose basis of Specht vectors is naturally indexed by the standard Young tableaux of shape $(n,n)$ \cite{PPR, MR2869443, MR2801314}.  

Relating different bases of symmetric group representations is an extremely fertile field with a long tradition in symmetric function theory \cite[Sections I.6 and III.6]{MR1354144}. More recently Kazhdan-Lusztig and others compared the Springer basis to the Kazhdan-Lusztig basis, and in fact the Kazhdan-Lusztig polynomials encode the transition matrix between those bases \cite{MR1994220, MR560412, MR597198}. Garsia-McLarnan relate the Kazhdan-Lusztig and Specht bases, showing the change-of-basis matrix with respect to a certain ordering of basis elements is upper-triangular with ones along the diagonal \cite[Theorem 5.3]{garsia1988relations}. In the context of this paper, the Kazhdan-Lusztig and web bases coincide \cite{frenkel1998kazhdan}, but this is not true in general \cite{housley2015robinson}. Webs in our context also coincide with the dual canonical bases for certain tensor products of $U_q(\mathfrak{sl_2})$ representations, but again this is not true of webs for $\mathfrak{sl}_k$ when $k\geq 3$ \cite{MR1446615, KK}.  

This paper asks: how are the web basis for $\mathfrak{sl}_2$ and the Specht basis related? This question has been indirectly addressed via the study of the Kazhdan-Lusztig and dual canonical bases and their relationships to the Specht and similar bases \cite{garsia1988relations, rhoades2011bitableaux}. The perspective of webs gives us new information about change-of-basis coefficients  in the $\mathfrak{sl}_2$ case and provides a foundation for studying more general web and Specht bases.

To relate these bases, we construct a web graph and a tableau graph with vertices of webs and standard Young tableaux respectively.  Two standard tableaux have an edge between them in the tableau graph if one can be obtained from the other by exchanging $i$ and $i+1$.  This graph also appears in geometry, where it encodes information about the components of Springer fibers \cite[Definitions 6.6 and 6.7]{pagnon2007adjacency} and in mathematical biology, where it encodes matchings of complementary base pairs of DNA \cite{seegerer2014rna}.  Similar but different graphs occur naturally in other contexts, including combinatorics (for instance the Bruhat graph, see for an overview \cite[Chapter 2]{MR2133266}) and geometry (where it characterizes geometric relationships between components of Springer fibers \cite{fresse2009betti}). 

The web graph requires more technical background to describe explicitly so we defer details to Section \ref{section: preliminaries}.  Intuitively the web representation can be thought of as coming from first twisting two adjacent strands in the web and then resolving according to traditional skein relations \cite[Theorem 5.9]{MR2869443}.  The web graph has a directed edge from one web $w$ to another $w'$ if this twisting-and-resolving process results in the combination $w+w'$ and if $w'$ has more nested arcs than $w$.  

Our first main result is Theorem \ref{theorem: graph isomorphism}, which proves that the web graph and the tableau graph are isomorphic as directed graphs via the traditional bijection between webs and standard tableaux \cite{MR1994220}.  We then use earlier results about the tableau graph \cite{pagnon2007adjacency} to conclude that the graphs are directed, acyclic graphs with a unique greatest element and a unique least element.  Thus the reachability condition within the graph induces a poset structure on both webs and tableaux.

The language of tableaux better illuminates some properties of the graph while the language of webs is better suited to demonstrate other properties.  For instance Section \ref{section: fundamental properties} shows that the tableau graph is a subgraph of the Bruhat graph.  From this perspective Corollary \ref{corollary: distance is number of inversions} proves that distance in the tableau graph is measured by the number of inversions in certain permutations corresponding to tableaux.  From the other perspective Section \ref{section: nesting} describes a way to compute distance graphically using webs: distance from the origin is measured by the amount of nesting within a web.  This allows us to interpret the partial order in terms of nesting and unnesting, which occupies the rest of Section \ref{section: nesting}.

As an application, we study the transition matrix between the web basis and the Specht basis.  To construct this matrix, we explicitly give an $S_{2n}$-equivariant isomorphism from the Specht module to web space, and in fact prove in Corollary \ref{corollary: unique equivariant isomorphism}
 that this equivariant isomorphism is essentially unique.  In an earlier paper, the authors constructed a map of representations from web space to the Specht module \cite[Lemma 4.2]{MR2801314}, so we conclude that the previous map is the inverse of the one defined in this paper. 

Our final results identify the Specht basis with its image in web space under this equivariant isomorphism and then analyze the transition matrix between the Specht and web bases.  Theorem \ref{theorem: upper triangular} proves that the transition matrix is upper-triangular with ones along the diagonal,  which Garsia-McLarnan proved earlier using the Kazhdan-Lusztig basis \cite{garsia1988relations}. Using our framework, we strengthen this result in Theorem \ref{theorem: upper triangular and respects poset} which shows that in addition the $(w,w')$ entry is zero unless $w \preceq w'$ according to the partial order induced by the web graph.

We make two conjectures that are supported by all current data: 1) For all $w\preceq w'$ the $(w,w')$ entry is nonzero, and 2) all entries in the transition matrix are nonnegative.  We also conjecture that similar results hold for webs for $\mathfrak{sl}_3$.  

The second author was partially supported by NSF DMS-1362855.  The authors gratefully acknowledge helpful comments from Sabin Cautis, Mikhail Khovanov, Greg Kuperberg,  Brendon Rhoades, Anne Schilling, and John Stembridge.

\section{Preliminaries on Webs and Tableaux}\label{section: preliminaries}

\subsection{Young tableaux and Specht modules}

Let $m\in \mathbb{N}$ and let $\lambda = (n_1 \geq n_2 \geq \ldots \geq n_k)$ denote a partition of $m$.  A {\em Young diagram} of shape $\lambda$ is a top- and left- justified collection of boxes where the $i^{th}$ row has $n_i$ boxes. In this paper we focus on the case where $m=2n$ for some positive integer $n$ and where the partition $\lambda = (n,n)$. In this case Young diagrams of shape $\lambda$ have two rows of equal length.

A {\em Young tableau} of shape $\lambda$ is a filling of the Young diagram of shape $\lambda$ with the numbers $\{1, \ldots, m\}$ such that every number occurs exactly once. If the numbers increase strictly from left to right along rows and from top to bottom along columns, the tableau is {\em standard}. Denote by $\mathcal{T}_n$ the set of standard Young tableaux of shape $(n,n)$.  We remark that the number of elements of $\mathcal{T}_n$ is the $n^{th}$ Catalan number $C_n = \frac{1}{n+1} { 2n \choose n}$. 

Let $S_{2n}$ be the symmetric group consisting of permutations of $2n$ elements. Denote by $s_i$ the permutation that transposes $i$ and $i+1$ and leaves all other values fixed. It is natural to permute entries of tableaux as follows: for each Young tableau $T$ define $s_i\cdot T$ to be the Young tableau obtained by swapping entries $i$ and $i+1$ in $T$. However, depending on the entries in $T$ it is possible that $T\in \mathcal{T}_n$ while $s_i\cdot T \notin \mathcal{T}_n$. In other words this action does not preserve the subset of standard Young tableaux. 

Specht modules were developed in essence to turn this into a well-defined action of the symmetric group.  Consider the complex vector space $V^{\mathcal{T}_n}$ with basis indexed by the set of standard Young tableaux $\mathcal{T}_n$. Given $T\in \mathcal{T}_n$ we denote the corresponding vector in  $V^{\mathcal{T}_n}$ by $v_T$. The vector space $V^{\mathcal{T}_n}$ is called the Specht module for the partition $(n,n)$.  

The Specht module $V^{\mathcal{T}_n}$ is actually a subspace of a larger vector space generated by equivalence classes of tableaux called tabloids.  A {\em tabloid} is an equivalence class of tableaux up to reordering of rows.  Denote the tabloid containing a specific tableau $T$ by $\{T\}$. The Specht vector $v_T$ corresponding to a (not necessarily standard) tableau $T$ is defined as
\[v_T = \sum_{\sigma \in Col(T)} sign(\sigma) \{\sigma \cdot T\} \]
where $Col(T)$ is the group of permutations that reorder the columns of $T$. The Specht module $V^{\mathcal{T}_n}$ is the span of the Specht vectors, and it has basis given by the vectors corresponding to standard tableaux. 

The natural symmetric group action on tableaux described above preserves tabloids and thus acts on the Specht module. In fact the Specht module is irreducible under this symmetric group action.  We will need the following facts about this action. 
\begin{itemize}
\item Let $T$ be a standard Young tableau of shape $(n,n)$ and let $\sigma\in S_{2n}$ be a permutation for which $\sigma \cdot T$ is also a standard Young tableau.  Then the action on Specht vectors satisfies 
\begin{equation} \label{equation: acting on Specht vectors}
\sigma\cdot v_T=v_{\sigma\cdot T}.
\end{equation}
(See e.g. \cite[Exercise 3, page 86]{MR1464693}.)  We remark that if $s_i\cdot T\notin \mathcal{T}_n$ then it is still true that $\sigma \cdot v_T =  v_{\sigma \cdot T}$ but the vector $v_{\sigma \cdot T}$ is no longer a basis vector of the Specht module.
\item Let $T_0$ be the standard Young tableau of shape $(n,n)$ whose top row is filled with the odd numbers $1,3,5,\ldots,2n-1$ and whose bottom row is filled with the even numbers $2,4,6,\ldots,2n$.  Then for each odd $i=1,3,5,\ldots,2n-1$
\begin{equation} \label{equation: acting on v_{T_0}}
s_i \cdot v_{T_0} = -v_{T_0}.
\end{equation}  
To see this, observe that if $T_0$ is the column-filled tableau defined above then $s_i \in Col(T)$, and so the claim holds.
\end{itemize}

\subsection{Webs}

Next we describe webs and the $S_{2n}$-action on web space.  A {\em web} on $2n$ points is a nonintersecting arrangement of arcs above a horizontal axis connecting $2n$ vertices on the axis. We enumerate the vertices from left to right and reference arcs according to their endpoints, so $(i,j)$ is in the web $w$ if $i<j$ and vertices $i$ and $j$ are connected by an arc in $w$. We sometimes know the endpoints of an arc but not their order, in which case we use the notation $\{i,j\}$. The figures typically do not number the vertices in webs. Denote the set of all webs on $2n$ points by $\mathcal{W}_n$.  Like standard Young tableaux of shape $(n,n)$ the set $\mathcal{W}_n$ is enumerated by the Catalan numbers. 

Let $V^{\mathcal{W}_n}$ be the complex vector space with basis indexed by the set of webs $\mathcal{W}_n$. Given a web $w\in \mathcal{W}_n$ we denote the corresponding vector in $V^{\mathcal{W}_n}$ by $w$. This differs from the notational convention for vectors in $V^{\mathcal{T}_n}$ because we consider permutations acting both on the set $\mathcal{T}_n$ and the vector space $V^{\mathcal{T}_n}$ and these actions only sometimes coincide. Indeed only the Specht module $V^{\mathcal{T}_n}$ admits a true group action.  By contrast there is only one action of $S_{2n}$ on the vector space $V^{\mathcal{W}_n}$.

To define this action choose $w\in V^{\mathcal{W}_n}$ and $s_i\in S_{2n}$. The image $s_i \cdot w$ is:
\[ s_i \cdot w = \left\{ \begin{array}{ll} -w & \textup{ if $(i,i+1)$ is an arc in $w$, and otherwise } \\
 & \hspace{1em} \\
                               w+w' & \begin{array}{l}\textup{ where $w'$ differs from $w$ only in the two arcs incident to $i$ and $i+1$,} \\ \textup{ with $w$ containing $\{i,j\}$ and $\{i+1,k\}$ and $w'$ containing $(i,i+1)$ and $\{j,k\}$} \end{array} \end{array} \right. \]
 Note that the relative order of $i, j,$ and $k$ determines if these arcs are nested or unnested. 
 
With respect to this action $V^{\mathcal{W}_n}$ is an irreducible representation and in fact coincides with the representation on the Specht module \cite{PPR, MR2801314}.  One can also construct and study this representation via skew Howe duality \cite{cautis2014webs}.

\subsection{The tableau graph and the web graph}

We now construct two graphs, one with vertices given by standard Young tableaux and the other with vertices given by webs.  Edges will come from the $S_{2n}$-action on each object.  We will then prove that these two graphs are isomorphic as directed graphs.
 
 Let $\Gamma^{\mathcal{T}_n}$ be the graph whose vertex set is $\mathcal{T}_n$ with an edge between $T$ and $T'$ if and only if there is a simple transposition such that $s_i\cdot T = T'$.  We typically consider the graph to have labels on the edges, and label the edge between $T$ and $T'$ by $s_i$ if $s_i \cdot T = T'$.  We also consider the directed graph with same vertex set and edge set as $\Gamma^{\mathcal{T}_n}$ but with each edge directed $T \stackrel{s_i}{\longrightarrow} T'$ if $i$ is below $i+1$ in $T$.  We use the same notation for both graphs, and call each the {\em tableau graph}.  Figure \ref{figure: tableau and web graphs} shows $\Gamma^{\mathcal{T}_3}$.

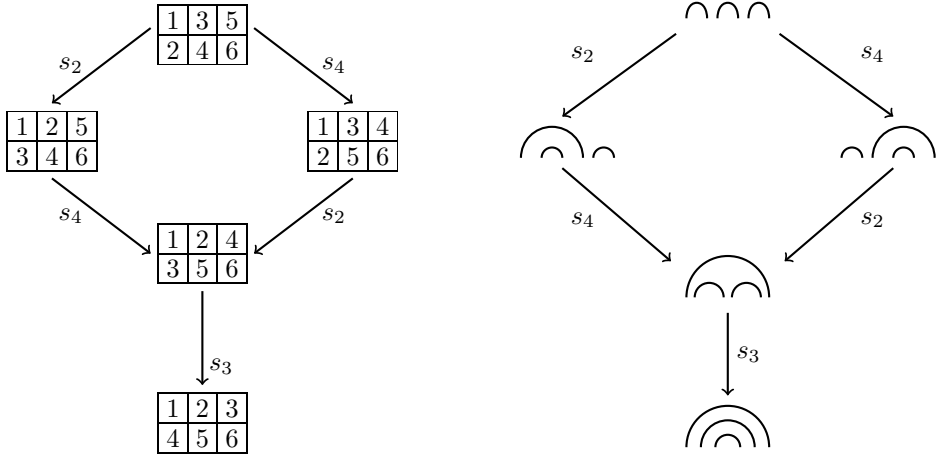
\begin{figure}[h]

\begin{tikzpicture}[baseline=0cm, scale=0.25]
\node at (0,1) {$\young(123,456)$};
\node at (0,10) {$\young(124,356)$};
\node at (-8,16) {$\young(125,346)$};
\node at (8,16) {$\young(134,256)$};
\node at (0,21.65) {$\young(135,246)$};
\draw[style=thick, ->] (0,8)--(0,3);
\draw[style=thick, ->] (8,14)--(2.75,10);
\draw[style=thick, ->] (-8,14)--(-2.75,10);
\draw[style=thick, <-] (8,18)--(2.75,22);
\draw[style=thick, <-] (-8,18)--(-2.75,22);
\node at (1,4) {$s_3$};
\node at (7,12) {$s_2$};
\node at (-7,12) {$s_4$};
\node at (7,20) {$s_4$};
\node at (-7,20) {$s_2$};
\end{tikzpicture}
\hspace{.5in} 
\raisebox{-10pt}{\begin{tikzpicture}[baseline=0cm, scale=0.275]
\draw[style=thick] (-2,1) to[out=90, in=180] (0,3) to[out=0, in=90] (2,1);
\draw[style=thick] (-1.3,1) to[out=90, in=180] (0,2.3) to[out=0, in=90] (1.3,1);
\draw[style=thick] (-.6,1) to[out=90, in=180] (0,1.6) to[out=0, in=90] (.6,1);

\draw[style=thick] (-2,8.25) to[out=90, in=180] (0,10.25) to[out=0, in=90] (2,8.25);
\draw[style=thick] (-1.6,8.25) to[out=90, in=180] (-.9,8.95) to[out=0, in=90] (-.2,8.25);
\draw[style=thick] (.2,8.25) to[out=90, in=180] (.9,8.95) to[out=0, in=90] (1.6,8.25);

\draw[style=thick] (-10,15) to[out=90, in=180] (-8.5, 16.5) to[out=0, in=90] (-7, 15);
\draw[style=thick] (-9, 15) to[out=90, in=180] (-8.5, 15.5) to[out=0, in=90] (-8, 15);
\draw[style=thick] (-6.5, 15) to[out=90, in=180] (-6, 15.5) to[out=0, in=90] (-5.5, 15);

\draw[style=thick] (10,15) to[out=90, in=0] (8.5, 16.5) to[out=180, in=90] (7, 15);
\draw[style=thick] (9, 15) to[out=90, in=0] (8.5, 15.5) to[out=180, in=90] (8, 15);
\draw[style=thick] (6.5, 15) to[out=90, in=0] (6, 15.5) to[out=180, in=90] (5.5, 15);

\draw[style=thick] (-.5, 21.75) to[out=90, in=180] (0, 22.5) to[out=0, in=90] (.5, 21.75);
\draw[style=thick] (-2, 21.75) to[out=90, in=180] (-1.5, 22.5) to[out=0, in=90] (-1, 21.75);
\draw[style=thick] (2, 21.75) to[out=90, in=0] (1.5, 22.5) to[out=180, in=90] (1, 21.75);

\draw[style=thick, ->] (0,7.5)--(0,3.5);
\draw[style=thick, ->] (8,14.5)--(2.75,10);
\draw[style=thick, ->] (-8,14.5)--(-2.75,10);
\draw[style=thick, <-] (8,17)--(2.5,21);
\draw[style=thick, <-] (-8,17)--(-2.5,21);
\node at (1,5.5) {$s_3$};
\node at (7,12) {$s_2$};
\node at (-7,12) {$s_4$};
\node at (7,20) {$s_4$};
\node at (-7,20) {$s_2$};
\end{tikzpicture}}
\caption{The $(3,3)$ tableau and web graphs} \label{figure: tableau and web graphs}
\end{figure}

Let $\Gamma^{\mathcal{W}_n}$ be the directed graph with vertex set $\mathcal{W}_n$ and a labeled, directed edge $w\stackrel{s_i}{\longrightarrow}w'$ between $w,w'\in \mathcal{W}_n$ if and only if both 
\begin{itemize}
\item $s_i \cdot w = w+w'$ and 
\item $w$ has unnested arcs $(j,i)$ and $(i+1, k)$. 
\end{itemize}
(The first condition is a necessary condition for the edge to exist and the second determines its direction.) We call $\Gamma^{\mathcal{W}_n}$ the {\em web graph}.  Figure \ref{figure: tableau and web graphs} shows $\Gamma^{\mathcal{W}_3}$.

\begin{remark}
Note that the target endpoint $w'$ of an edge $w \stackrel{s_i}{\longrightarrow} w'$ in the web graph must contain the arc $(i,i+1)$.  Moreover the arc $(i,i+1)$ is nested under the arc $(j,k)$.
\end{remark}

We use an explicit combinatorial bijection $\psi: \mathcal{T}_n\rightarrow \mathcal{W}_n$ between standard Young tableaux and webs that arises from geometric considerations \cite{MR1994220}.  Given $T\in \mathcal{T}_n$ the corresponding web $\psi(T)$ is constructed by the following process. Begin with $2n$ evenly spaced vertices on a horizontal axis. Starting with the leftmost number in the bottom row of $T$, connect the corresponding vertex to its nearest unoccupied lefthand neighbor. Repeat this process moving from left to right across the bottom row of $T$.  For example, recall that $T_0$ is the standard tableau whose $i^{th}$ column is filled with $2i-1$ and $2i$ for each $i$ and define $w_0$ to be the web with arcs $(1,2), (3,4), \ldots, (2n-1, 2n)$. Then 
\[\psi(T_0) = w_0.\]
Figure \ref{figure: tableau and web graphs} shows more examples: the map $\psi$ sends each vertex in the tableau graph to the corresponding vertex in the web graph.  

In fact the map $\psi$ induces an isomorphism between the directed graphs $\Gamma^{\mathcal{T}_n}$ and $\Gamma^{\mathcal{W}_n}$ as we now prove.

\begin{theorem} \label{theorem: graph isomorphism}
The directed graphs $\Gamma^{\mathcal{W}_n}$ and $\Gamma^{\mathcal{T}_n}$ are isomorphic. Furthermore the isomorphism is given by the bijection $\psi$ between tableaux and webs described above.
\end{theorem}

\begin{proof}
Suppose $T$ and $T'$ are two standard tableaux.  Let $w$ be the web that corresponds to $T$ and $w'$ be the web that corresponds to $T'$.  The following are equivalent by the definitions:
\begin{itemize}
\item entry $i$ is below $i+1$ in $T$
\item $i$ is a right endpoint and $i+1$ is a left endpoint in $w$
\item the web $w$ has a pair of unnested arcs $(j,i)$ and $(i+1, k)$
\end{itemize}  
There is an edge $T \stackrel{s_i}{\longrightarrow} T'$ in the tableau graph if and only if $s_i\cdot T = T'$ and $i$ is below $i+1$ in  $T$.  By above this is true if and only if $w$ and $w'$ are identical except that $w$ contains arcs $(j,i)$ and $(i+1,k)$ while $w'$ contains arcs $(j,k)$ and $(i,i+1)$.  In other words, this is equivalent to the existence of the edge $w  \stackrel{s_i}{\longrightarrow} w'$ in the web graph.  

We conclude that $T \stackrel{s_i}{\longrightarrow} T'$ if and only if $w \stackrel{s_i}{\longrightarrow} w'$ and therefore $\Gamma^{\mathcal{W}_n}$ and $\Gamma^{\mathcal{T}_n}$ are isomorphic as directed, labeled graphs via the bijection between tableaux and webs.
\end{proof}

Since the graphs are isomorphic, results from other papers on the tableau graph apply to the web graph \cite{pagnon2007adjacency, seegerer2014rna}. For instance, the following result collects several properties of the tableau graph that we will need.

\begin{proposition}[Pagnon-Ressayre]
The web and tableau graphs are directed, acyclic graphs with a unique source and a unique sink.
\end{proposition}

Figure \ref{figure: tableau and web graphs} shows that the corresponding undirected graph has cycles whenever $n$ is large enough.  Note also that this proposition means the graph represents the Hasse diagram of a poset.  Section \ref{section: nesting} studies this poset structure graphically using webs.

While the two graphs are isomorphic, we retain the language and notation of both graphs throughout this manuscript for several reasons: 1) certain properties are easier to see using one interpretation and not the other; 2) each graph has its own ``hidden properties" that the other doesn't (e.g. the tableau graph can be thought of as an undirected graph because each $s_i$ acts as an involution when defined on a tableau).

\section{Fundamental properties of the web and tableau graph}\label{section: fundamental properties}

In this section we describe the graph $\Gamma^{\mathcal{T}_n}$.  One key tool is viewing the tableau graph as a subgraph of the Bruhat graph; from that we obtain a way to measure distance in the tableau graph as the number of inversions of certain permutations.

First we identify explicitly the unique source and unique sink in the tableau (or web) graph.  This unique source is the heart of the proof of representation theoretic Theorem \ref{T0 to W0}.

\begin{lemma}
The column-filled tableau $T_0$ is the unique source in $\Gamma^{\mathcal{T}_n}$.  The corresponding web $w_0$ with arcs $(1,2), (3,4), \ldots, (2n-1,2n)$ is the unique source for $\Gamma^{\mathcal{W}_n}$.

The tableau whose first row is filled with $\{1,2,3,\ldots,n\}$ is the unique sink in $\Gamma^{\mathcal{T}_n}$.  The corresponding web has arcs $(1,2n), (2,2n-1), \ldots,(n,n+1)$ and is the unique sink for $\Gamma^{\mathcal{W}_n}$.
\end{lemma}

\begin{proof}
We prove that $w_0$ is a source using the web graph.  All edges incident to $w_0$ in $\Gamma^{\mathcal{W}_n}$ are directed out since $w_0$ contains no nested arcs.  

We identify the sink via its tableau.  Note that the tableau with $\{1,2,3,\ldots,n\}$ along the first row is incident to just one edge, namely the edge that exchanges $n$ and $n+1$.  Since $n$ is on the top row we conclude that the tableau is the endpoint of the directed edge and is thus a sink.
\end{proof}

In the next few lemmas, we describe local properties of the graphs $\Gamma^{\mathcal{W}_n}$ and $\Gamma^{\mathcal{T}_n}$.  The first says that consecutive edges in a path have different labels.

\begin{lemma}\label{consecutive}
Suppose that $T\stackrel{s_i}{\longrightarrow} T' \stackrel{s_j}{\longrightarrow} T''$ in the tableau graph. Then $i\neq j$.
\end{lemma}
\begin{proof}
By definition of edges in $\Gamma^{\mathcal{T}_n}$ we know $i$ is in the first row and $i + 1$ is in the second row in $T'$. Hence the edge from $T'$ to $T''$ cannot have label $s_i$.
\end{proof}

Similarly edges that are incident to the same tableau must have different labels, which is refined in the next lemma.

\begin{lemma}
Let $T$ be a tableau. Suppose that $T \stackrel{s_i}{\longrightarrow} T'$ and $T \stackrel{s_j}{\longrightarrow} T''$ are two distinct edges in the tableau graph. Then $|i-j|>1$.
\end{lemma}

\begin{proof}
Two distinct edges directed out of $T$ must have different labels by definition, say $s_i$ and $s_j$ for $i \neq j$.  The definition of $s_i$ and $s_j$ implies that $i$ and $j$ are in the second row of $T$ but $i+1$ and $j+1$ are in the first row. Therefore $i \neq j+1$ and $j \neq i+1$ and so $|i-j|>1$.
\end{proof}

This leads immediately to the following corollary about the structure of the web and tableau graphs. 

\begin{corollary}\label{diamond}
For each pair of edges $T \stackrel{s_i}{\longrightarrow} T'$ and $T \stackrel{s_j}{\longrightarrow} T''$ in the tableau graph, there is a diamond formed by two directed paths $T \stackrel{s_i}{\longrightarrow} T'\stackrel{s_j}{\longrightarrow} T'''$ and $T \stackrel{s_j}{\longrightarrow} T''\stackrel{s_i}{\longrightarrow} T'''$ as shown in Figure \ref{figure: diamond lemma}.
\end{corollary}

\begin{proof}
If $s_i \cdot T$ and $s_j \cdot T$ are both standard and $|i-j|>1$ then $s_is_jT=s_js_iT$ is also standard.
\end{proof}

\begin{figure}[h]
\begin{tikzpicture}[scale=0.2]

\node at (0,9.5) {\Large{$T'''$}};
\node at (.12,21.25) {\Large{$T$}};
\node at (-6,15.75) {\Large{$T'$}};
\node at (6.1,15.75) {\Large{$T''$}};

\node at (5,12) {$s_i$};
\node at (-5,12) {$s_j$};
\node at (4,20) {$s_j$};
\node at (-4,20) {$s_i$};

\draw[style=thick, ->] (5.5,14.5)--(1.5,10);
\draw[style=thick, ->] (-5.5,14.5)--(-1.5,10);
\draw[style=thick, <-] (5.5,17)--(1,21);
\draw[style=thick, <-] (-5.5,17)--(-1,21);
\end{tikzpicture}
\caption{A diamond in the graph $\Gamma^{\mathcal{T}_n}$ }\label{figure: diamond lemma}
\end{figure}
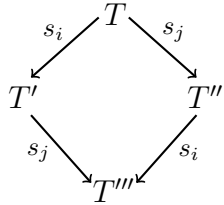

The previous claim is a kind of ``diamond lemma" for web and tableau graphs, specializing a similar lemma in the Bruhat graphs of permutations \cite[Proposition 2.2.7]{MR2133266}.  This is not accidental: the tableau graph is in fact a subgraph of the well-known Bruhat graph, as we discuss in the rest of this section.

Given a tableau $T$ we can construct a permutation $\sigma_T$ in one-line notation by reading the entries of $T$ in a specific order (see e.g. \cite[Definition 3.5 and subsequent]{dewitt2012poset}).  For our purposes one particular reading order is most useful, as defined next.

\begin{definition}\label{definition: permutation of tableau}
For each tableau $T \in \mathcal{T}_n$ construct a permutation $\sigma_T$ by reading the entries of $T$ in the order that $T_0$ is filled, namely down each column from the leftmost to the rightmost column.  
\end{definition}

For example the tableau $T_0$ corresponds to the identity permutation $1 \hspace{1em} 2\hspace{1em}3\hspace{1em}4 \cdots 2n$ while the word corresponding to the sink in the tableau graph is $1 \hspace{1em}(n+1) \hspace{1em}2 \hspace{1em}(n+2)\hspace{1em}3 \hspace{1em}(n+3) \cdots n\hspace{1em} (2n)$.  

With this description, the following proposition is immediate.

\begin{proposition}
The tableau graph is a subgraph of the Bruhat graph.  The vertices consist of those permutations that correspond to standard Young tableaux of shape $(n,n)$ and the edges consist of those edges $(i,i+1)$ that connect vertices $\sigma \leftrightarrow (i,i+1) \sigma$ in the Bruhat graph. 
\end{proposition} 

Note that this is not an induced graph: there are edges corresponding to non-simple reflections in the Bruhat graph (and even in the Hasse diagram of Bruhat order) that do not appear in the tableau graph.  

\begin{remark}
The above discussion is a concise version of \cite[Theorem 8.1]{pagnon2007adjacency}.  The bijection between permutations and tableau used by Pagnon-Ressayre relates to ours by change of base point, since they do not associate the identity permutation to $T_0$ \cite[Theorem 8.9]{pagnon2007adjacency}. 
\end{remark}

Recall that the number of {\em inversions} of a permutation in one-line notation is the number of instances of a pair $\cdots i \cdots j \cdots$ with $i>j$.  Also recall that a product $s_{j_1} s_{j_2} \cdots s_{j_k}$ of simple reflections is called {\em reduced} if there is no way to write the same permutation as a product of fewer simple reflections.

\begin{lemma}\label{lemma: min-length path in graph}
If there is a directed edge $T \stackrel{s_{i}}{\longrightarrow} T'$ in the tableau graph then the number of inversions of $\sigma_T$ is one fewer than the number of inversions of $\sigma_{T'}$.  

Suppose that $T \stackrel{s_{i_1}}{\longrightarrow} T_1\stackrel{s_{i_2}}{\longrightarrow} T_2 \cdots \stackrel{s_{i_k}}{\longrightarrow} T_k$ is a directed path in the tableau graph.  Then $s_{i_1} s_{i_2} \cdots s_{i_k}$ is a reduced word in the permutation group.
\end{lemma}

\begin{proof}
If an edge labeled $s_i$ is directed out of the tableau $T$ then $i$ is in the second row of $T$ and $i+1$ is in the first row.  Since $T$ is standard $i$ must be in a column to the left of the column of $i+1$ and so $i$ is to the left of $i+1$ in $\sigma_T$.  Thus the pair $i,i+1$ do not contribute an inversion in $\sigma_T$ and by the same argument they do in $\sigma_{T'}$.  The rest of the entries are the same so the first claim holds.  We can restate the second claim as follows: 
\begin{itemize}
\item the permutation $\sigma_{T_k}$ has $k$ more inversions than $\sigma_T$ and
\item the permutations satisfy $\sigma_{T_k} = s_{i_1} s_{i_2} \cdots s_{i_k} \sigma_T$.  
\end{itemize}
The permutation $s_{i_1} s_{i_2} \cdots s_{i_k}$ thus has $k$ inversions and so is a reduced word \cite[Section 1.4, Proposition 1.5.2]{MR2133266}.
\end{proof}

This result says that directed paths in the tableau graph correspond to reduced words for permutations.  However note that there could be reduced words for a permutation that can't be constructed from paths in the tableau graph, since some simple reflections label edges that are not in the tableau graph.

We close this section by noting that distance in the graph is counted by the inversions in $\sigma_T$.

\begin{corollary}\label{corollary: distance is number of inversions}
For each tableau $T$ in the graph, the distance $dist(T_0,T)$ between $T$ and $T_0$ is the same as the number of inversions of $\sigma_T$.
\end{corollary}

\begin{proof}
Let $T$ be a tableau.  The distance $dist(T_0,T)$ is the length of a directed path from $T_0 \stackrel{\sigma}{\longrightarrow} T$.  By Lemma \ref{lemma: min-length path in graph} this is also the length of the permutation $\sigma$ which is by definition $\sigma_T$.
\end{proof}







\section{Nesting and the partial order in the web graph} \label{section: nesting}

This section explores the partial order from the perspective of the web graph.  The main results prove that the partial order (and the statistics that arise from it, including distance from the source) can be interpreted easily in terms of a graphical property of webs.  This property, called nesting, is defined next.

\begin{definition}\label{definition: nesting}
For each arc $a$ in a web $w$, let $n(a)$ be the number of arcs that $a$ is nested beneath in $w$. Define the nesting number $n(w)$ of a web to be the sum of the values $n(a)$ taken over all arcs in $w$. 
\end{definition}

The next result shows how nesting numbers change when exactly two arcs are nested or unnested in a particular web. In the proof, we use the notion of an umbrella arc defined as follows. 

\begin{definition}
Let $a=(r,s)$ be an arc in the web $w$.  The arc $b=(r',s')$ is the umbrella arc of $a$ if $r'<r<s<s'$ and there is no arc $c=(r'',s'')$ with $r'<r''<r<s<s''<s'$.  If there is no such arc, we say $a$ has no umbrella arc.
\end{definition}

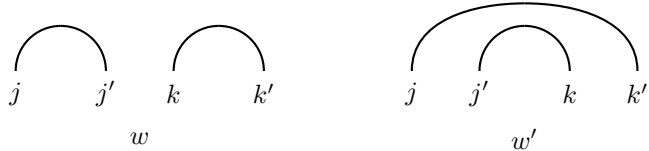
\begin{figure}[h]
\begin{tikzpicture}[baseline=0cm, scale=0.3]
\draw[style=thick] (4,0) to[out=90, in=0] (2,2);
\draw[style=thick] (2,2) to[out=180, in=90] (0,0);
\draw[style=thick] (11,0) to[out=90, in=0] (9,2);
\draw[style=thick] (9,2) to[out=180, in=90] (7,0);
\node at (0,-1) {$j$};
\node at (4,-1) {$j'$};
\node at (7,-1) {$k$};
\node at (11,-1) {$k'$};
\node at (5.5,-3) {{\bf $w$}};
\end{tikzpicture}
\hspace{.5in} 
\begin{tikzpicture}[baseline=0cm, scale=0.3]
\draw[style=thick] (10,0) to[out=90, in=0] (5,3);
\draw[style=thick] (5,3) to[out=180, in=90] (0,0);
\draw[style=thick] (7,0) to[out=90, in=0] (5,2);
\draw[style=thick] (5,2) to[out=180, in=90] (3,0);
\node at (0,-1) {$j$};
\node at (3,-1) {$j'$};
\node at (7,-1) {$k$};
\node at (10,-1) {$k'$};
\node at (5,-3) {{\bf $w'$}};
\end{tikzpicture}
\caption{$n(w')=n(w) + k-j'$}\label{figure:general nesting}
\end{figure}

\begin{lemma}\label{lemma: changing one nest or unnest}
Suppose $w$ is a web with arcs $(j,j')$ and $(k,k')$ where $j'<k$, and $w'$ is the web with arcs $(j,k')$ and $(j',k)$ that is otherwise identical to $w$. Then $n(w') = n(w) + k-j' \geq n(w)+1$.
\end{lemma}

\begin{proof}
If there were an arc in $w$ with exactly one endpoint between $j'$ and $k$ then this arc would cross $(j,k')$ in $w'$ violating our hypotheses. It follows that $(j,j')$ and $(k,k')$ in $w$ and $(j,k')$ in $w'$ all have the same umbrella arc and so 
\[n_w((j,j'))=n_w((k,k'))=n_{w'}((j,k')).\] 
Observe that $(j',k)$  has umbrella arc $(j,k')$ in $w'$ and hence $n_{w'}((j', k)) =n_{w'}((j,k'))+1$. 

Now let $a=(r,s)$ be an arc common to $w$ and $w'$. We compare $n_w(a)$ and $n_{w'}(a)$.
\begin{itemize}
\item If $a$ is to the left of $j$ (namely $r<s<j$), to the right of $k'$ (namely $k'<r<s$), or above the arcs (namely $r<j<k'<s$) then $n_w(a)=n_{w'}(a)$. 
\item If $a$ is beneath arc $(j,j')$ in $w$ then $j<r<s<j'$ and hence $a$ is beneath arc $(j,k')$ but not $(j',k)$ in $w'$. Thus $n_w(a)=n_{w'}(a)$. If $k<r<s<k'$ then by a symmetric argument $n_w(a)=n_{w'}(a)$. 
\item If $a$ is between arcs $(j,j')$ and $(k,k')$ in $w$ then $j'<r<s<k'$. Therefore $a$ has the same umbrella arc as $(j,j')$ and $(k,k')$ in $w$ and so 
\[n_w(a) = n_w((j,j'))=n_w((k,k')).\] 
However $a$ has umbrella arc $(j',k)$ in $w'$ so $n_{w'}(a) = n_w(a)+2$.
\end{itemize}
Summing over all arcs in $w'$ we conclude that 
\[n(w') = n(w) + 1 + 2\left(\frac{k'-j-1}{2}\right) = n(w) + k'-j\]
\end{proof}

This leads to a collection of small but important corollaries about how the nesting number changes in the web graph.

\begin{corollary}\label{corollary: nesting number grows by one}
Suppose that $w \rightarrow w'$ is an edge in the web graph.  Then $n(w')=n(w)+1$.
\end{corollary}

\begin{proof}
Since $w \rightarrow w'$ is an edge in the web graph, there is a simple transposition $s_i$ with $s_i\cdot w = w+w'$. Furthermore the webs $w$ and $w'$ are identical except that $w$ has arcs $(j,i)$ and $(i+1, k)$ while $w'$ has arcs $(j,k)$ and $(i,i+1)$.  The result follows from Lemma \ref{lemma: changing one nest or unnest}.
\end{proof}

It follows that distance from the web $w_0$ is the same as the nesting number.

\begin{corollary}
In the web graph the distance and nesting number are the same, and both are the same as the number of inversions of the permutation corresponding to $T$ in Definition \ref{definition: permutation of tableau}:
\[\textup{dist}(w_0, w) = n(w) = \# \textup{ inversions of } \sigma_{\psi^{-1}(w)}\]
\end{corollary}

\begin{proof}
We prove this inductively. The base case is true since $n(w_0) = 0$.  Consider a web $w'$ that is distance $k$ from $w_0$. Then there is an edge $w \rightarrow w'$ from a web $w$ whose distance from $w_0$ is $n(w)=k-1$ by the inductive hypothesis. Corollary \ref{corollary: nesting number grows by one} and Corollary \ref{corollary: distance is number of inversions} complete the proof.
\end{proof}

Moreover there are only two ways the nesting number can change when acting on a web by $s_i$.

\begin{corollary} \label{corollary: nesting options across one edge}
Assume that $s_i\cdot w = w+w'$ where $w\neq w'$. Then either $n(w)+1=n(w')$ or $n(w)>n(w')$. 
\end{corollary}

\begin{proof}
Note that $(i,i+1)$ is not an arc in $w$ since otherwise $s_i\cdot w = -w$. Thus $w$ has two distinct arcs incident on $i$ and $i+1$, say with (unordered) endpoints $\{i,j\}$ and $\{i+1, j'\}$. 

If $j<i<i+1<j'$ then the arcs are unnested. Thus there is an edge $w\rightarrow w'$ in the web graph labeled $s_i$ and so the nesting number increases as $n(w)+1 = n(w')$ by Corollary \ref{corollary: nesting number grows by one}. 

If not then either $j'<j<i<i+1$ or $i<i+1<j'<j$ since the arcs are noncrossing. This time Lemma \ref{lemma: changing one nest or unnest} implies that the nesting number drops.  
\end{proof}

In fact the previous result could be refined since Lemma \ref{lemma: changing one nest or unnest} says that the nesting number drops by one if $j=i-1$  and $j'<j<i<i+1$, by one if $j'=i+2$ and $i<i+1<j'<j$, and by at least three for any other in any other case when $j'<j<i<i+1$ or $i<i+1<j'<j$.

The next collection of results will further analyze the partial order $\preceq$ and show that certain natural web properties (e.g. two arcs are nested or unnested) can be interpreted directly in terms of $\preceq$. First we show that in the web graph at most one edge labeled $s_i$ is directed into a given web $w$. 

\begin{lemma}
Let $w\in \mathcal{W}_n$ with arc $(i,i+1)$. If $(i,i+1)$ has an umbrella arc then there is exactly one web $w'$ for which there is an edge $w' \rightarrow w$ labeled $s_i$. If $(i,i+1)$ has no umbrella arc then there is no edge labeled $s_i$ directed into $w$. 
\end{lemma}

\begin{proof}
The endpoint of any edge labeled $s_i$ is a web containing the arc $(i,i+1)$ where the nesting number of arc $(i,i+1)$ is at least one, so if $(i,i+1)$ has no umbrella arc then there is no edge $w' \rightarrow w$ labeled $s_i$.

If $(i,i+1)$ has an umbrella arc $(j,k)$ then consider the web $w'$ with arcs $(j,i)$ and $(i+1,k)$ that otherwise agrees with $w$.  On the one hand $s_i \cdot w' = w'+w$ by construction.  On the other hand suppose that $T'$ and $T$ are the tableaux corresponding to $w'$ and $w$ respectively and consider the tableau graph, which is isomorphic to the web graph.  The map $s_i$ is defined on $T$ because the fact that $s_i \cdot T' =T$ implies that $i$ and $i+1$ are in different rows and columns of $T$.  The map $s_i$ is a well-defined involution on its domain so exactly one edge incident to $T$ in the tableau graph is labeled $s_i$.  This proves the claim in this case.
\end{proof}

We can also count the number of ways of acting by $s_i$ to reduce nesting number and obtain a web $w$. We think of these as ``invisible backward edges" in the graph $\Gamma^{\mathcal{W}_n}$ as illustrated in Figure \ref{figure: detecting invisible edges}.

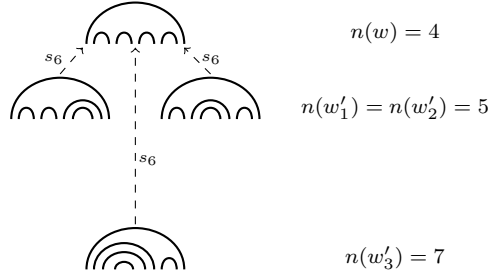
\begin{figure}[h]
\begin{tikzpicture}[scale=0.2]
\draw[style=thick] (-.5, 21.75) to[out=90, in=180] (0, 22.5) to[out=0, in=90] (.5, 21.75);
\draw[style=thick] (-2, 21.75) to[out=90, in=180] (-1.5, 22.5) to[out=0, in=90] (-1, 21.75);
\draw[style=thick] (2, 21.75) to[out=90, in=0] (1.5, 22.5) to[out=180, in=90] (1, 21.75);
\draw[style=thick] (2.5, 21.75) to[out=90, in=180] (3,22.5) to[out=0, in=90] (3.5, 21.75);
\draw[style=thick] (-2.5, 21.75) to[out=90, in=180] (.75, 24.5) to[out=0, in=90] (4, 21.75);

\draw[style=thick] (-5.5, 16.75) to[out=90, in=180] (-5, 17.5) to[out=0, in=90] (-4.5, 16.75);
\draw[style=thick] (-7, 16.75) to[out=90, in=180] (-6.5, 17.5) to[out=0, in=90] (-6, 16.75);
\draw[style=thick] (-4, 16.75) to[out=90, in=180] (-2.75, 18) to[out=0, in=90] (-1.5, 16.75);
\draw[style=thick] (-3.4, 16.75) to[out=90, in=180] (-2.75,17.5) to[out=0, in=90] (-2.1, 16.75);
\draw[style=thick] (-7.5, 16.75) to[out=90, in=180] (-4.25, 19.5) to[out=0, in=90] (-1, 16.75);

\draw[style=thick] (3, 16.75) to[out=90, in=180] (3.5, 17.5) to[out=0, in=90] (4, 16.75);
\draw[style=thick] (4.5, 16.75) to[out=90, in=180] (5.75, 18) to[out=0, in=90] (7, 16.75);
\draw[style=thick] (5.1, 16.75) to[out=90, in=180] (5.75, 17.5) to[out=0, in=90] (6.4, 16.75);
\draw[style=thick] (7.5, 16.75) to[out=90, in=180] (8,17.5) to[out=0, in=90] (8.5, 16.75);
\draw[style=thick] (2.5, 16.75) to[out=90, in=180] (5.75, 19.5) to[out=0, in=90] (9, 16.75);

\draw[style=thick] (-.6, 6.75) to[out=90, in=180] (0, 7.25) to[out=0, in=90] (.6, 6.75);
\draw[style=thick] (-2, 6.75) to[out=90, in=180] (0, 8.5) to[out=0, in=90] (2, 6.75);
\draw[style=thick] (-1.4, 6.75) to[out=90, in=180] (0, 7.9) to[out=0, in=90] (1.4, 6.75);
\draw[style=thick] (2.5, 6.75) to[out=90, in=180] (3,7.5) to[out=0, in=90] (3.5, 6.75);
\draw[style=thick] (-2.5, 6.75) to[out=90, in=180] (.75, 9.5) to[out=0, in=90] (4, 6.75);

\node at (18, 22.5) {\Small{$n(w)=4$}};
\node at (18, 17.5) {\Small{$n(w'_1)=n(w'_2)=5$}};
\node at (18, 7.5) {\Small{$n(w'_3)=7$}};

\draw[style=dashed, ->] (-4.25, 19.75) -- (-2.75, 21.5);
\draw[style=dashed, ->] (5.75, 19.75) -- (4, 21.5);
\draw[style=dashed, ->] (.75, 9.75)--(.75,21.5);

\node at (-4.65, 20.75) {\tiny{$s_6$}};
\node at (5.8, 20.75) {\tiny{$s_6$}};
\node at (1.6, 14) {\tiny{$s_6$}};
\end{tikzpicture}
\caption{There are three webs with $s_6\cdot w'_i = w'_i + w$ and $n(w_i')>n(w)$.}\label{figure: detecting invisible edges}
\end{figure}

\begin{lemma}\label{unnest-below}
Let $w\in \mathcal{W}_n$ with arc $(i,i+1)$. The number of webs $w'$ such that $s_i\cdot w' = w' + w$ and $n(w')>n(w)$ is the number of arcs in $w$ with the same umbrella arc as $(i,i+1)$. 
\end{lemma}

\begin{proof}
Suppose that $(j,k)$ has the same umbrella arc as $(i,i+1)$ in $w$. It follows that either $k<j<i<i+1$ or $i<i+1<k<j$. In other words $(j,k)$ is either to the left of $i$ or to the right of $i+1$.  Define $w'$ to be the web with arcs $(j,i+1)$ and $(k,i)$ that otherwise agrees with $w$.  By construction $s_i \cdot w' = w'+w$ and $n(w')>n(w)$.  

Now suppose that $(j,k)$ does not have the same umbrella arc as $(i,i+1)$ in $w$.  Then there is an arc in $w$ with exactly one endpoint between $(j,k)$ and $(i,i+1)$.  In that case there is no way to change only the arcs incident to the endpoints $j,k,i,i+1$ in $w$ without violating the noncrossing condition, so there is no web $w'$ that differs from $w$ only on the endpoints $j,k,i,i+1$.  This proves the claim.
\end{proof}

\begin{figure}[h]
\begin{tikzpicture}[scale=0.2]
\draw[style=thick] (0,0) to[out=90, in=180] (3,3);
\draw[style=thick](3,3) to[out=0, in=90] (6,0);
\draw[style=thick] (7,0) to[out=90, in=180] (10,3);
\draw[style=thick](10,3) to[out=0, in=90] (13,0);
\draw[fill=gray] (1,0)--(5,0)--(3,2)--(1,0);
\draw[fill=white] (8,0)--(12,0)--(10,2)--(8,0);
\node at (5.8,-.5) {\tiny{$i$}};
\node at (7.6,-.56) {\tiny{$i\!+\!1$}};
\node at (6.8,-3) {{\bf $w$}};
\end{tikzpicture}
 \raisebox{3pt}{$\stackrel{\cdots}{\longrightarrow}$}
\begin{tikzpicture}[scale=0.2]
\draw[style=thick] (0,0) to[out=90, in=180] (6.5,4);
\draw[style=thick](6.5,4) to[out=0, in=90] (13,0);
\draw[style=thick] (1,0) to[out=90, in=180] (4,3);
\draw[style=thick](4,3) to[out=0, in=90] (7.25,0);
\draw[fill=gray] (2,0)--(6,0)--(4,2)--(2,0);
\draw[fill=white] (8,0)--(12,0)--(10,2)--(8,0);
\node at (5.9,-.5) {\tiny{$i$}};
\node at (7.7,-.56) {\tiny{$i\!+\!1$}};
\node at (6.8,-3) {{\bf $w'$}};
\end{tikzpicture}
\hspace{.5in}
\begin{tikzpicture}[scale=0.2]
\draw[style=thick] (0,0) to[out=90, in=180] (3,3);
\draw[style=thick](3,3) to[out=0, in=90] (6,0);
\draw[style=thick] (7,0) to[out=90, in=180] (10,3);
\draw[style=thick](10,3) to[out=0, in=90] (13,0);
\draw[fill=gray] (1,0)--(5,0)--(3,2)--(1,0);
\draw[fill=white] (8,0)--(12,0)--(10,2)--(8,0);
\node at (5.9,-.5) {\tiny{$i$}};
\node at (7.5,-.56) {\tiny{$i\!+\!1$}};
\node at (6.8,-3) {{\bf $w$}};
\end{tikzpicture}
 \raisebox{3pt}{$\stackrel{\cdots}{\longrightarrow}$}
\begin{tikzpicture}[scale=0.2]
\draw[style=thick] (0,0) to[out=90, in=180] (6.5,4);
\draw[style=thick](6.5,4) to[out=0, in=90] (13,0);
\draw[style=thick] (6,0) to[out=90, in=180] (9,3);
\draw[style=thick](9,3) to[out=0, in=90] (12,0);
\draw[fill=gray] (1,0)--(5,0)--(3,2)--(1,0);
\draw[fill=white] (7,0)--(11,0)--(9,2)--(7,0);
\node at (5.9,-.5) {\tiny{$i$}};
\node at (7.7,-.56) {\tiny{$i\!+\!1$}};
\node at (6.8,-3) {{\bf $w'$}};
\end{tikzpicture}
\caption{A path exists from $w$ to $w'$ in $\Gamma^{\mathcal{W}_n}$, so $w\preceq w'$.}\label{figure: expanding umbrella}
\end{figure}
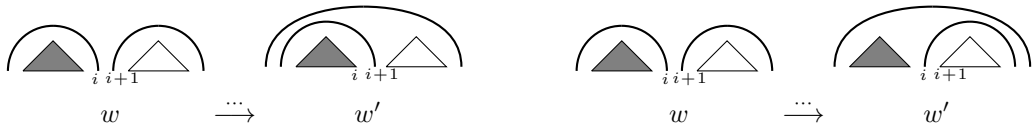

In the next lemma we show that it is always possible to pull an arc over a sub-web diagram, as shown in Figure \ref{figure: expanding umbrella}.

\begin{lemma}\label{arcdrag}
Suppose that $w'$ is obtained from $w$ by expanding one arc to form an umbrella over a larger subdiagram, as in Figure \ref{figure: expanding umbrella}. Then $w\preceq w'$.
\end{lemma}

\begin{proof}
The argument is the same whether we expand the umbrella to the left or to the right, so we consider only the case of expanding to the left as shown on the left in Figure \ref{figure: expanding umbrella}.

We proceed by induction on the number of arcs in the web $w$. If there are exactly two arcs then $w$ has arcs $(1,2)$ and $(3,4)$ while $w'$ has arcs $(1,4)$ and $(2,3)$. In this case $w\stackrel{s_i}{\longrightarrow} w'$ is an edge in the web graph so $w\preceq w'$. 

Assume the claim holds for all webs with $n$ arcs, and let $w, w'$ be webs with $n+1$ arcs arranged as shown on the left of Figure \ref{figure: expanding umbrella}. Consider the intermediate web $w''$ that is identical to $w$ except $w''$ has arcs $(1,2n+2)$ and $(i,i+1)$ whereas $w$ has arcs $(1, i)$ and $(i+1, 2n+2)$ as shown in Figure \ref{figure: intermediate step in arcdrag}.
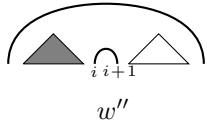
\begin{figure}[h]
\begin{tikzpicture}[scale=0.2]
\draw[style=thick] (0,0) to[out=90, in=180] (6.5,4);
\draw[style=thick](6.5,4) to[out=0, in=90] (13,0);
\draw[style=thick] (5.75,0) to[out=90, in=180] (6.5,1);
\draw[style=thick](6.5,1) to[out=0, in=90] (7.25,0);
\draw[fill=gray] (1,0)--(5,0)--(3,2)--(1,0);
\draw[fill=white] (8,0)--(12,0)--(10,2)--(8,0);
\node at (5.7,-.5) {\tiny{$i$}};
\node at (7.4,-.55) {\tiny{$i\!+\!1$}};
\node at (7,-3) {{\bf $w''$}};
\end{tikzpicture}
\caption{There is an edge $w\stackrel{s_i}{\longrightarrow} w''$ in $\Gamma^{\mathcal{W}_n}$.}\label{figure: intermediate step in arcdrag}
\end{figure}
There is an edge $w\stackrel{s_i}{\longrightarrow} w''$ in the web graph. Use the inductive hypothesis on the collection of arcs on vertices $2, \ldots, i+1$ to obtain a directed path in the web graph from $w''$ to $w'$. We conclude $w\preceq w'$ as desired.
\end{proof}

We use the previous lemma to show that {\em any} pair of webs that differ in only two arcs have a directed path between them in the web graph.

\begin{theorem}\label{theorem: nesting of arcs implies directed path}
Suppose that $w$ and $w'$ differ only in the arcs incident to vertices $j< j'< k< k'$ and suppose that $w$ contains arcs $(j,j')$ and $(k,k')$ while $w'$ contains arcs $(j,k')$ and $(j',k)$ as shown in Figure \ref{figure: single nesting/unnesting}. Then $w\preceq w'$.
\end{theorem}

\begin{proof}
The proof follows from repeated applications of Lemma \ref{arcdrag}. Figure \ref{figure: single nesting/unnesting} shows the key intermediate steps. First apply Lemma \ref{arcdrag} to find a path from $w$ to the web $w_1$ obtained by moving the left endpoint of arc $(k,k')$ over the arcs between $(j,j')$ and $(k,k')$. Next there is an edge labeled $s_{j'}$ from $w_1$ to $w_2$ by construction.  Finally $w'$ is obtained from $w_2$ by expanding the arc $(j',j'+1)$ to form an umbrella over the shaded region indicated in Figure \ref{figure: single nesting/unnesting}, so Lemma \ref{arcdrag} guarantees a directed path from $w_2$ to $w'$.  Thus there is a directed path from $w$ to $w'$ as desired.
\end{proof}
\begin{figure}[h]
\begin{tikzpicture}[baseline=0cm, scale=0.28]
\draw[style=thick] (4,0) to[out=90, in=0] (2,2);
\draw[style=thick] (2,2) to[out=180, in=90] (0,0);
\draw[style=thick] (11,0) to[out=90, in=0] (9,2);
\draw[style=thick] (9,2) to[out=180, in=90] (7,0);
\draw[fill=gray] (4.5,0)--(6.5,0)--(5.5, 1.5)--(4.5,0);
\draw[fill=white] (.5,0)--(3.5,0)--(2,1.5)--(.5,0);
\draw[fill=black] (7.5,0)--(10.5,0)--(9,1.5)--(7.5,0);
\node at (0,-1) {$j$};
\node at (4,-1) {$j'$};
\node at (7,-1) {$k$};
\node at (11,-1) {$k'$};
\node at (5.5,-3) {{\bf $w$}};
\end{tikzpicture}
\raisebox{5pt}{ $\preceq$}
\begin{tikzpicture}[baseline=0cm, scale=0.28]
\draw[style=thick] (4,0) to[out=90, in=0] (2,2);
\draw[style=thick] (2,2) to[out=180, in=90] (0,0);
\draw[style=thick] (11,0) to[out=90, in=0] (8,2);
\draw[style=thick] (8,2) to[out=180, in=90] (5,0);
\draw[fill=gray] (5.5,0)--(7.5,0)--(6.5, 1.5)--(5.5,0);
\draw[fill=white] (.5,0)--(3.5,0)--(2,1.5)--(.5,0);
\draw[fill=black] (7.75,0)--(10.5,0)--(9,1.5)--(7.75,0);
\node at (0,-1) {$j$};
\node at (4,-1) {$j'$};
\node at (7,-1) {$k$};
\node at (11,-1) {$k'$};
\node at (5.5,-3) {{\bf $w_1$}};
\end{tikzpicture}
\raisebox{5pt}{$\stackrel{s_{j'}}{\longrightarrow}$}
\begin{tikzpicture}[baseline=0cm, scale=0.28]
\draw[style=thick] (11,0) to[out=90, in=0] (5.5,3);
\draw[style=thick] (5.5,3) to[out=180, in=90] (0,0);
\draw[style=thick] (5,0) to[out=90, in=0] (4.5,1);
\draw[style=thick] (4.5,1) to[out=180, in=90] (4,0);
\draw[fill=gray] (5.5,0)--(7.5,0)--(6.5, 1.5)--(5.5,0);
\draw[fill=white] (.5,0)--(3.5,0)--(2,1.5)--(.5,0);
\draw[fill=black] (7.75,0)--(10.5,0)--(9,1.5)--(7.75,0);
\node at (0,-1) {$j$};
\node at (4,-1) {$j'$};
\node at (7,-1) {$k$};
\node at (11,-1) {$k'$};
\node at (5.5,-3) {{\bf $w_2$}};
\end{tikzpicture}
\raisebox{5pt}{$\preceq$}
\begin{tikzpicture}[baseline=0cm, scale=0.28]
\draw[style=thick] (11,0) to[out=90, in=0] (5.5,3);
\draw[style=thick] (5.5,3) to[out=180, in=90] (0,0);
\draw[style=thick] (7,0) to[out=90, in=0] (5.5,2);
\draw[style=thick] (5.5,2) to[out=180, in=90] (4,0);
\draw[fill=gray] (4.5,0)--(6.5,0)--(5.5, 1.5)--(4.5,0);
\draw[fill=white] (.5,0)--(3.5,0)--(2,1.5)--(.5,0);
\draw[fill=black] (7.5,0)--(10.5,0)--(9,1.5)--(7.5,0);
\node at (0,-1) {$j$};
\node at (4,-1) {$j'$};
\node at (7,-1) {$k$};
\node at (11,-1) {$k'$};
\node at (5.5,-3) {{\bf $w'$}};
\end{tikzpicture}
\caption{Path between $w$ and $w'$ in the web graph}\label{figure: single nesting/unnesting}
\end{figure}
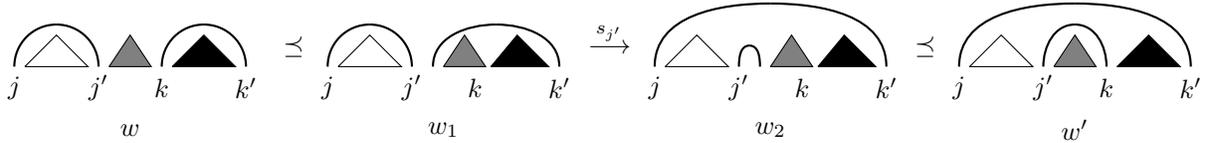

\begin{remark}
Khovanov uses a partial order on the set $\mathcal{W}_n$ to relate knot invariants to the cohomology of $(n,n)$ Springer varieties. His partial order is determined via a directed graph with edges $w \rightarrow w'$ whenever $w$ and $w'$ are related as in Figure \ref{figure:general nesting}, so the edge set in $\Gamma^{\mathcal{W}_n}$ is a subset of the edge set in Khovanov's graph. Theorem \ref{theorem: nesting of arcs implies directed path} implies Khovanov's partial order and the one described here coincide.   Distance in Khovanov's setting is in general shorter because there are more edges in his graph and distance is defined via the shortest {\it undirected} path.
\end{remark}

\section{An equivariant map relating the Specht and web bases}
Let $V^{\mathcal{T}_n}$ and $V^{\mathcal{W}_n}$ be the complex $S_{2n}$ representations generated by complex linear combinations of Specht vectors of shape $(n,n)$ and webs of shape $(n,n)$ respectively. In this section we define a map $\phi: V^{\mathcal{T}_n} \rightarrow V^{\mathcal{W}_n}$ that is equivariant with respect to the actions in Section \ref{section: preliminaries} and then study its properties. We begin with an observation about any such map: it must send the Specht vector $v_{T_0}$ to the web $w_0$. 

\begin{theorem}\label{T0 to W0}
Let $\phi:V^{\mathcal{T}_n} \rightarrow V^{\mathcal{W}_n}$ be a linear map. If $\phi$ is $S_{2n}$-equivariant then $\phi\left(v_{T_0}\right) = aw_0$ for some $a\in \mathbb{C}$.
\end{theorem}
\begin{proof}
Let $\phi\left(v_{T_0}\right) = \sum_{w\in \mathcal{W}_n} c_w w$ with $c_w\in \mathbb{C}$. Let $i$ be odd and consider $w\in \mathcal{W}_n$. If $(i,i+1)\in w$ then $s_i\cdot c_w w = -c_w w$.  If not then $s_i \cdot w = w+w'$ for some web $w'$ containing the arc $(i,i+1)$. Partitioning $\mathcal{W}_n$ into webs that contain the arc $(i,i+1)$ and webs that do not, we see
\begin{align*}
 s_i \cdot \phi\left(v_{T_0}\right) &=  \sum_{\stackrel{w\in \mathcal{W}_n}{(i,i+1)\in w}} s_i \cdot c_w w  + \sum_{\stackrel{w\in \mathcal{W}_n}{(i,i+1)\notin w}} s_i \cdot c_w w\\
 &= \sum_{\stackrel{w\in \mathcal{W}_n}{(i,i+1)\in w}} -c_w w  + \sum_{\stackrel{w\in \mathcal{W}_n}{(i,i+1)\notin w}} c_w w + \sum_{\stackrel{w\in \mathcal{W}_n}{(i,i+1)\in w}} b_w w
 \end{align*}
 where the third sum in the final step above collects webs with arc $(i,i+1)$ that came from acting on a web without arc $(i,i+1)$. (Note that $b_w$ is one of the coefficients from the middle sum, reindexed.)  Recall from Equation \eqref{equation: acting on v_{T_0}} that $s_i\cdot v_{T_0} = -v_{T_0}$.  Since $\phi$ is equivariant we know that $s_i \cdot \phi\left(v_{T_0}\right) = -\phi\left(v_{T_0}\right)$. Considering the middle sum in the final step above, we conclude that $c_w = 0$ for all $w\in \mathcal{W}_n$ with $(i,i+1)\notin w$. The equation above holds for all odd $i$ and so $c_w = 0$ unless $w$ contains arcs $(1,2), (3,4),\ldots, (2n-1, 2n)$. The only web containing all of these arcs is $w_0$.  We conclude that $\phi\left(v_{T_0}\right) = aw_0$ for some $a\in \mathbb{C}$ as desired.
\end{proof}

This motivates the following definition of a linear map between the Specht module and web space.

\begin{definition}\label{mapdef}
Let $T_0$ be the column-filled tableau of shape $(n,n)$ and $w_0$ the corresponding web. Define the map $\phi: V^{\mathcal{T}_n} \rightarrow V^{\mathcal{W}_n}$ to be $\phi(v_T) = \sigma\cdot w_0 $ where $\sigma = s_{i_1} \cdots s_{i_k}$ labels a directed path from $T_0$ to $T$ in the tableau graph. 
\end{definition}

We first confirm that the map is in fact well-defined.

\begin{lemma}
The map $\phi$ is well-defined. i.e. $\phi(v_T)$ is independent of choice of path $\sigma$ from $T_0$ to $T$.
\end{lemma}

\begin{proof}
Since the tableau graph is connected, there is always a path from $T_0$ to an arbitrary tableau $T$. Say that $\sigma = s_{i_1}\cdots s_{i_k}$ and $\sigma' = s_{j_1} \cdots s_{j_k}$ label two directed paths from $T_0$ to $T$ in the tableau graph. This means that $\sigma\cdot T_0 = \sigma' \cdot T_0 = T$ by definition of the tableau graph.  This means the Specht vectors satisfy $\sigma\cdot v_{T_0} = \sigma' \cdot v_{T_0} = v_T$ by Equation \eqref{equation: acting on Specht vectors}.

The tableau graph is a subgraph of the Bruhat graph and hence $\sigma = \sigma'$ as permutations. Then we have $\sigma \cdot w_0=\sigma' \cdot w_0$ since $V^{\mathcal{W}_n}$ is a representation of $S_{2n}$. We conclude that $\phi(v_{T})$ is independent of path chosen.
\end{proof}

This allows us to prove that the map $\phi$ is equivariant with respect to the $S_{2n}$ actions defined in Section \ref{section: preliminaries}.  In fact this essentially means that $\phi$ is the only equivariant isomorphism between these two representations.

\begin{corollary}\label{corollary: unique equivariant isomorphism}
The map $\phi$ is $S_{2n}$-equivariant and is (up to scaling) the only equivariant isomorphism $V^{\mathcal{T}_n} \rightarrow V^{\mathcal{W}_n}$.
\end{corollary}
\begin{proof}
Observe that $\phi(v_{T_0}) = w_0$. By Theorem \ref{T0 to W0} and the definition of $\phi$ we conclude that any equivariant map on $V^{\mathcal{T}_n}$ agrees with $\phi$ on the basis of $V^{\mathcal{T}_n}$. There is at least one equivariant map between $V^{\mathcal{T}_n}$ and $V^{\mathcal{W}_n}$ because the two representations are isomorphic.  Since a linear map is determined by what it does on a basis, we conclude $\phi$ must be equivariant and thus is (up to scaling) the only equivariant isomorphism between $V^{\mathcal{T}_n}$ and $V^{\mathcal{W}_n}$.
\end{proof}

Let $T'$ be a standard tableau with corresponding web $w'$. Then, using notation similar to the proof of Theorem \ref{T0 to W0}, we can express $\phi(v_{T'})$ as follows:  $$\phi(v_{T'}) = \sum_{w\in \mathcal{W}_n} c_w^{w'} w.$$ Our goal is to describe the coefficients $c_w^{w'}$ coming from the map $\phi$. The following theorem also follows from work of Garsia-McLarnan relating the Kazhdan-Lusztig basis to Young's natural basis \cite[Theorem 5.3]{garsia1988relations}.

\begin{theorem}\label{theorem: upper triangular}
The map $\phi$ is upper-triangular with ones along the diagonal. 
\end{theorem}

\begin{proof}
Complete the partial order on webs coming from the web graph to a total order $<$ in a way that also respects nesting numbers, meaning that if $n(\tilde{w})<n(w')$ then the web vectors satisfy $\tilde{w}<w'$.   Consider the matrix for $\phi$ with respect to that ordering of the web basis.  

Now let $T'$ be a standard tableau. It is sufficient to show that for $\phi(v_{T'})$ the following statements hold: 
\begin{enumerate}
\item  If $c_w^{w'} \neq 0$ then $n(w)\leq n(w')$ with $n(w)=n(w')$ if and only if $w=w'$ and 
\item $c_{w'}^{w'} = 1$. 
\end{enumerate}
We proceed via induction on the distance from $T_0$ or equivalently the nesting number of the associated web. The base case is true by definition since $\phi(v_{T_0}) = w_0$.

Assume the result is true for all tableaux of distance less than $k$ from $T_0$ and let $T'$ be a tableau with distance $k$ from $T_0$.  This means there is an edge $T'' \stackrel{s_i}{\longrightarrow} T'$ in the tableau graph from a tableau $T''$ at distance $k-1$ from $T_0$. By the inductive hypothesis $$\phi(v_{T''}) = \sum_{w\in \mathcal{W}_n} c_w^{w''} w$$ with the coefficients $c_w^{w''}$ satisfying properties (1) and (2).  Moreover by definition of $\phi$ we have 
\[\phi(v_{T'})= s_i \cdot \phi(v_{T''}) =  \sum_{w\in \mathcal{W}_n} c_w^{w''} s_i \cdot w.\]

Consider $w\in \mathcal{W}_n$ such that $c_w^{w''}\neq 0$ and $w\neq w''$. Then $n(w)<n(w'')$ by hypothesis. After acting by $s_i$ either $s_i\cdot w = -w$ or $s_i\cdot w = w+\bar{w}$ for some web $\bar{w}$. Corollary \ref{corollary: nesting options across one edge} implies that $\bar{w}$ has nesting number at most $n(w'')$ and hence strictly smaller than $n(w')$.  Thus all webs that come from terms other than $w''$ in $\phi(v_{T''})$ are strictly less than $w'$ in the order coming from the web graph. 

Now consider the action of $s_i$ on $w''$. By assumption $s_i\cdot w'' = w''+w'$. We conclude both that $w'$ is the only web with nesting number $n(w')$ with a nonzero coefficient in $\phi(v_{T'})$ and that its coefficient is $c_{w''}^{w''}=1$ as desired.
\end{proof}

We now refine this result: if $c_w^{w'}\neq 0$ then not only is $w \leq w'$ but in fact $w\preceq w'$ in the partial order defined by the directed web graph.  Our argument will proceed much like the previous result.

We begin with a lemma that, loosely speaking, stacks together the ``diamond lemma" repeatedly (namely Corollary \ref{diamond}).  


\begin{lemma} \label{diamondapp}
Assume $w\preceq w'''$ with $w'''\stackrel{s_i}{\longrightarrow} w'$ and $w\stackrel{s_i}{\longrightarrow} w''$. Then $w''\preceq w'$.
\end{lemma}
\begin{proof}
Let $w$, $w'$, $w''$, and $w'''$ be as in the statement above. Since $w \preceq w'''$ there is a path $$w \stackrel{s_{j_1}}{\longrightarrow} w_1 \stackrel{s_{j_2}}{\longrightarrow} w_2 \stackrel{s_{j_3}}{\longrightarrow} \cdots \stackrel{s_{j_k}}{\longrightarrow} w''' \stackrel{s_i}{\longrightarrow} w'$$ in the web graph.  We prove that there is also a path in the web graph from $w''$ to $w'$ using induction on $k = d(w, w''')$.

If $k=0$ then $w=w'''$. Since edges come from a well-defined symmetric group action on $V^{\mathcal{W}_n}$, there is at most one edge labeled $s_i$ directed from $w=w'''$ in the web graph. Since $w''' \stackrel{s_i}{\longrightarrow} w'$ and $w\stackrel{s_i}{\longrightarrow} w''$ we have $w''=w'$ and so $w''\preceq w'$. 

If $k=1$ then $w \stackrel{s_{j_1}}{\longrightarrow}w'''\stackrel{s_i}{\longrightarrow} w'$.  Lemma \ref{consecutive} showed that consecutive edges in the web graph have different labels. Therefore $j_1 \neq i$ and so the edge $w \stackrel{s_i}{\longrightarrow} w''$ in the web graph is distinct from the edge $w \stackrel{s_{j_1}}{\longrightarrow}w'''$.  Lemma \ref{diamond} now shows that there is an edge $w''\stackrel{s_{j_1}}{\longrightarrow} w'$ and thus $w''\preceq w'$.

Assume $w''\preceq w'$ whenever $w$ is less than distance $k$ from $w'''$ and now suppose that the distance between $w$ and $w'''$ is $k$. The first edge $w \stackrel{s_{j_1}}{\longrightarrow} w_1$ in the path from $w$ to $w'''$ satisfies either $j_1 = i$ or $j_1 \neq i$.  If $j_1 = i$ then $w''=w_1$ and so $w''$ lies on a path from $w$ to $w'''$. Since $w'''$ has an edge to $w'$ we conclude $w''\preceq w'$.

Finally consider the case when $j_1\neq i$. By Corollary \ref{diamond} there exists a web ${\bar{w}}$ in the web graph along with edges $w''\stackrel{s_{j_1}}{\longrightarrow }{\bar{w}}$ and $w_1 \stackrel{s_i}{\longrightarrow}{\bar{w}}$. The distance between $w_1$ and $w'''$ is less than $k$ so we can use the inductive hypothesis to conclude ${\bar{w}}\preceq w'$. Since there is an edge from $w''$ to ${\bar{w}}$ we conclude $w''\preceq w'$.
\end{proof}

This brings us to the main theorem.

\begin{theorem}\label{theorem: upper triangular and respects poset}
If $c_{v}^{w'} \neq 0$ then there is a directed path from $v$ to $w'$ in the web graph, namely $v\preceq w'$.
\end{theorem}

\begin{proof}
The proof inducts on the distance between $w'$ and $w_0$ in the web graph. The base case is when the distance is zero, namely $w_0 \preceq w_0$.  Assume that if $c_v^{w'} \neq 0$ then $v\preceq w'$ for all $w'$ less than distance $k$ from $w_0$. Now assume $w'$ has distance $k$ from $w_0$ and consider any web $w'''$ with an edge $w''' \stackrel{s_i}{\longrightarrow} w'$ in the web graph.  As before we have
\[\phi(v_{T'})= s_i \cdot \phi(v_{T'''}) =  \sum_{\footnotesize \begin{array}{c}w\in  \mathcal{W}_n \\ w \preceq w'''\end{array}} c_w^{w'''} s_i \cdot w = \sum_{\footnotesize \begin{array}{c}w\in  \mathcal{W}_n \\ w \preceq w'''\end{array}} c_w^{w'''} (w + w'')\]
where we take $w''$ to be $-2w$ when $s_i \cdot w = -w$.

The two ways for a web to have nonzero coefficient in $\phi(v_{T'})$ correspond to the terms $w$ and $w''$ in the previous expression: either the coefficient of $w$ gets a contribution of $c_w^{w'''}$ from the first summand in one of the terms or the coefficient of $w''$ gets a contribution of $c_w^{w'''}$ from the second summand in one of the terms. In the first case (where $w$ plays the role of $v$) we have $w \preceq w'''$ by assumption and so since $w''' \longrightarrow w'$ we conclude $w \preceq w'$ as desired.  (If $w''=-2w$ this argument also applies.)

Now we decompose the second case (where $w''$ plays the role of $v$) even further: if $s_i \cdot w = w+w''$ then either $w \stackrel{s_i}{\longrightarrow}w''$ or $d(w, w_0)>d(w'', w_0)$ by Corollary \ref{corollary: nesting options across one edge}.  The first case is Lemma \ref{diamondapp} for which $w'' \preceq w'$ as desired.  

In the final case $w''$ and $w$ differ only in the arcs incident to vertices $j,i,i+1$, and $k$.  By comparing distance (equivalently nesting number) we know that $\{j,i\}$ and $\{k, i+1\}$ are in $w$ with $k<j<i$ or $i+1<k<j$ while $(k,j)$ and $(i,i+1)$ are in $w''$.  Thus $w'' \preceq w$ by Theorem \ref{theorem: nesting of arcs implies directed path}.  We conclude that $w'' \preceq w'$ as desired since $w''\preceq w\preceq w''' \rightarrow w'$.
\end{proof}

We conclude with the following conjecture, which incorporates the two conjectures from the Introduction.

\begin{conjecture}
For every tableau $T'$ with corresponding web $w'$, the vector $\phi(v_{T'})$ has the following form 
$$\phi(v_{T'}) = w' + \sum_{\stackrel{w\in \mathcal{W}_n}{w \prec w'}} c_w^{w'} v_w$$
where $c_w^{w'}>0$ for all $w$ in the sum.
\end{conjecture}

We also conjecture the main theorem of this paper extends to webs for $\mathfrak{sl}_3$ and perhaps all $\mathfrak{sl}_n$.

\bibliographystyle{plain}
\bibliography{references}

\def\cprime{$'$}
\begin{thebibliography}{10}

\bibitem{MR2133266}
Anders Bj{\"o}rner and Francesco Brenti.
\newblock {\em Combinatorics of {C}oxeter groups}, volume 231 of {\em Graduate
  Texts in Mathematics}.
\newblock Springer, New York, 2005.

\bibitem{cautis2008knot}
Sabin Cautis, Joel Kamnitzer, et~al.
\newblock Knot homology via derived categories of coherent sheaves, {I}: The
  $\mathfrak{sl}(2)$-case.
\newblock {\em Duke Mathematical Journal}, 142(3):511--588, 2008.

\bibitem{cautis2014webs}
{S}abin Cautis, {J}oel Kamnitzer, and {S}cott Morrison.
\newblock Webs and quantum skew {H}owe duality.
\newblock {\em Mathematische Annalen}, 360(1-2):351--390, 2014.

\bibitem{dewitt2012poset}
Barry Dewitt and Megumi Harada.
\newblock Poset pinball, highest forms, and $(n-2, 2) $ {S}pringer varieties.
\newblock {\em The Electronic Journal of Combinatorics}, 19(1):P56, 2012.

\bibitem{FKK}
Bruce Fontaine, Joel Kamnitzer, and Greg Kuperberg.
\newblock Buildings, spiders, and geometric {S}atake.
\newblock {\em Compos. Math.}, 149(11):1871--1912, 2013.

\bibitem{MR1446615}
Igor~B. Frenkel and Mikhail~G. Khovanov.
\newblock Canonical bases in tensor products and graphical calculus for
  {$U_q({\mathfrak s}{\mathfrak l}_2)$}.
\newblock {\em Duke Math. J.}, 87(3):409--480, 1997.

\bibitem{frenkel1998kazhdan}
Igor~B. Frenkel, Mikhail~G. Khovanov, and Alexandre~A. Kirillov~Jr.
\newblock Kazhdan-{L}usztig polynomials and canonical basis.
\newblock {\em Transformation groups}, 3(4):321--336, 1998.

\bibitem{fresse2009betti}
Lucas Fresse.
\newblock Betti numbers of {S}pringer fibers in type a.
\newblock {\em Journal of Algebra}, 322(7):2566--2579, 2009.

\bibitem{MR1464693}
William Fulton.
\newblock {\em Young tableaux}, volume~35 of {\em London Mathematical Society
  Student Texts}.
\newblock Cambridge University Press, Cambridge, 1997.
\newblock With applications to representation theory and geometry.

\bibitem{MR1994220}
Francis Y.~C. Fung.
\newblock On the topology of components of some {S}pringer fibers and their
  relation to {K}azhdan-{L}usztig theory.
\newblock {\em Adv. Math.}, 178(2):244--276, 2003.

\bibitem{garsia1988relations}
{Adriano M.} Garsia and {Timothy J.} McLarnan.
\newblock Relations between {Y}oung's natural and the {K}azhdan-{L}usztig
  representations of ${S}_n$.
\newblock {\em Advances in Mathematics}, 69(1):32--92, 1988.

\bibitem{housley2015robinson}
Matthew Housley, Heather~M Russell, and Julianna Tymoczko.
\newblock The {R}obinson--{S}chensted correspondence and {A}\_2-web bases.
\newblock {\em Journal of Algebraic Combinatorics}, 42(1):293--329, 2015.

\bibitem{jones1983index}
Vaughan~FR Jones.
\newblock Index for subfactors.
\newblock {\em Inventiones mathematicae}, 72(1):1--25, 1983.

\bibitem{jones1985polynomial}
Vaughan~{FR} Jones.
\newblock A polynomial invariant for knots via von {N}eumann algebras.
\newblock {\em Bulletin of the American Mathematical Society}, 12(1):103--111,
  1985.

\bibitem{kauffman1987state}
Louis~H Kauffman.
\newblock State models and the {J}ones polynomial.
\newblock {\em Topology}, 26(3):395--407, 1987.

\bibitem{MR560412}
David Kazhdan and George Lusztig.
\newblock Representations of {C}oxeter groups and {H}ecke algebras.
\newblock {\em Invent. Math.}, 53(2):165--184, 1979.

\bibitem{MR597198}
David Kazhdan and George Lusztig.
\newblock A topological approach to {S}pringer's representations.
\newblock {\em Adv. in Math.}, 38(2):222--228, 1980.

\bibitem{khovanov2004crossingless}
Mikhail Khovanov.
\newblock Crossingless matchings and the cohomology of (n, n) {S}pringer
  varieties.
\newblock {\em Communications in Contemporary Mathematics}, 6(04):561--577,
  2004.

\bibitem{KK}
Mikhail Khovanov and Greg Kuperberg.
\newblock Web bases for {${\rm sl}(3)$} are not dual canonical.
\newblock {\em Pacific J. Math.}, 188(1):129--153, 1999.

\bibitem{K}
Greg Kuperberg.
\newblock Spiders for rank {$2$} {L}ie algebras.
\newblock {\em Comm. Math. Phys.}, 180(1):109--151, 1996.

\bibitem{lickorish1992calculations}
W~BR Lickorish.
\newblock Calculations with the {T}emperley-{L}ieb algebra.
\newblock {\em Commentarii Mathematici Helvetici}, 67(1):571--591, 1992.

\bibitem{MR1354144}
I.~G. Macdonald.
\newblock {\em Symmetric functions and {H}all polynomials}.
\newblock Oxford Mathematical Monographs. The Clarendon Press Oxford University
  Press, New York, second edition, 1995.
\newblock With contributions by A. Zelevinsky, Oxford Science Publications.

\bibitem{pagnon2007adjacency}
NGJ Pagnon and Nicolas Ressayre.
\newblock Adjacency of {Y}oung tableaux and the {S}pringer fibers.
\newblock {\em Selecta Mathematica}, 12(3-4):517--540, 2007.

\bibitem{PPR}
T.~Kyle Petersen, Pavlo Pylyavskyy, and Brendon Rhoades.
\newblock Promotion and cyclic sieving via webs.
\newblock {\em J. Algebraic Combin.}, 30(1):19--41, 2009.

\bibitem{reshetikhin1991invariants}
Nicolai Reshetikhin and Vladimir~{G} Turaev.
\newblock Invariants of 3-manifolds via link polynomials and quantum groups.
\newblock {\em Inventiones mathematicae}, 103(1):547--597, 1991.

\bibitem{rhoades2011bitableaux}
Brendon Rhoades and Mark Skandera.
\newblock Bitableaux and zero sets of dual canonical basis elements.
\newblock {\em Annals of Combinatorics}, 15(3):499--528, 2011.

\bibitem{RTW}
G.~Rumer, E.~Teller, and H.~Weyl.
\newblock Eine f\"{u}r die valenztheorie geeignete basis der binaren
  vektorinvarianten.
\newblock {\em Nachr. Ges. Wiss. gottingen Math.-Phys. Kl.}, pages 499--504,
  1932.

\bibitem{MR2869443}
Heather~M. Russell.
\newblock A topological construction for all two-row {S}pringer varieties.
\newblock {\em Pacific J. Math.}, 253(1):221--255, 2011.

\bibitem{MR2801314}
Heather~M. Russell and Julianna~S. Tymoczko.
\newblock Springer representations on the {K}hovanov {S}pringer varieties.
\newblock {\em Math. Proc. Cambridge Philos. Soc.}, 151(1):59--81, 2011.

\bibitem{SCHÄFER_2012}
Gisa Sch\"{a}fer.
\newblock A graphical calculus for 2-block {S}paltenstein varieties.
\newblock {\em Glasgow Mathematical Journal}, 54(2):449--477, 2012.

\bibitem{seegerer2014rna}
Laura Seegerer, Jennifer Tripp, Julianna Tymoczko, and Judy Wang.
\newblock R{NA}, local moves on plane trees, and transpositions on tableaux.
\newblock {\em arXiv preprint arXiv:1411.3056}, 2014.

\bibitem{Stroppel-Webster}
Catharina Stroppel and Ben Webster.
\newblock 2-block {S}pringer fibers: {C}onvolution algebras and coherent
  sheaves.
\newblock {\em Comment Math Helv}, 87(2):477--520, 2012.

\bibitem{wehrli2009remark}
Stephan~M Wehrli.
\newblock A remark on the topology of (n, n) {S}pringer varieties.
\newblock {\em arXiv preprint arXiv:0908.2185}, 2009.

\end{thebibliography}

\end{document}